\documentclass[11pt]{article}
\usepackage{color}
\usepackage{graphicx}
\usepackage{subfigure}
\usepackage{epsfig}
\usepackage{cite}
\usepackage{amsmath, amssymb}
\usepackage{fancybox}
\usepackage{listings}
\usepackage{verbatim}
\usepackage{url}
\usepackage{esint}
\usepackage{fancyhdr}
\usepackage{booktabs}
\usepackage{caption}
\usepackage{dsfont}
\input{epsf.sty}
\usepackage{mathrsfs,amsthm,amsmath}
\newtheorem{theorem}{Theorem}[section]
\newtheorem{proposition}[theorem]{Proposition}

\newtheorem{lemma}{Lemma}[section]
\newtheorem{remark}{Remark}[section]

\newtheorem{condition}{Condition}
\begin{document}
\title{Asymptotic results for the absorption time of telegraph
processes with elastic boundary at the origin\thanks{The authors
acknowledge the support of: GNAMPA and GNCS groups of INdAM
(Istituto Nazionale di Alta Matematica); MIUR--PRIN 2017, Project
‘Stochastic Models for Complex Systems’ (no. 2017JFFHSH); MIUR
Excellence Department Project awarded to the Department of
Mathematics, University of Rome Tor Vergata (CUP E83C18000100006).}}
\author{Claudio Macci\thanks{Dipartimento di Matematica,
Universit\`a di Roma Tor Vergata, Via della Ricerca Scientifica,
00133 Rome, Italy. e-mail: \texttt{macci@mat.uniroma2.it}}\and
Barbara Martinucci\thanks{Dipartimento di Matematica,
Universit\`{a} degli Studi di Salerno, Via Giovanni Paolo II n.
132, 84084 Fisciano, SA, Italy. e-mail:
\texttt{bmartinucci@unisa.it}}\and Enrica
Pirozzi\thanks{Dipartimento di Matematica e Applicazioni,
Universit\`{a} di Napoli Federico II, Via Cintia, Complesso Monte
S. Angelo, 80126 Naples, Italy. e-mail:
\texttt{enrica.pirozzi@unina.it}}}
\date{}
\maketitle
\begin{abstract}
\noindent We consider a telegraph process with elastic boundary
at the origin studied recently in the literature (see e.g.
\cite{DicrescenzoMartinucciZacks}). It is a particular random
motion with finite velocity which starts at $x\geq 0$, and its
dynamics is determined by upward and downward switching rates
$\lambda$ and $\mu$, with $\lambda>\mu$, and an absorption
probability (at the origin) $\alpha\in(0,1]$. Our aim is to study
the asymptotic behavior of the absorption time at the origin
with respect to two different scalings: $x\to\infty$ in the
first case; $\mu\to\infty$, with $\lambda=\beta\mu$ for some
$\beta>1$ and $x>0$, in the second case. We prove several large
and moderate deviation results. We also present numerical estimates
of $\beta$ based on an asymptotic Normality result for the case of
the second scaling.\\

\end{abstract}

\section{Introduction}
The (integrated) telegraph process is an alternating random motion
with finite velocity, and has several applications in different
fields (for instance in physics, finance and mathematical biology).
The literature on the telegraph process and its generalizations
is quite large. In this paper we refer to a recent model with
elastic barrier studied in \cite{DicrescenzoMartinucciZacks}, where
it is possible find several references. Here we only recall some of
them.

We start with some references that studied the solution of the
telegraph equation; see e.g. \cite{Goldstein1951} and \cite{Kac1974}.
Among the first references that studied some probabilistic aspects,
we recall \cite{Orsingher1990} and \cite{Foong1992}. Moreover we also
cite \cite{Orsingher1995} and \cite{Ratanov1997} where the telegraph
process in the presence of reflecting and absorbing barriers was
investigated. Among the more recent references with some generalizations,
we recall \cite{StadjeZacks2004} for a model with random velocities,
\cite{DicrescenzoMartinucci2010} for a damped telegraph process,
\cite{Crimaldietal2013} for a model driven by certain random trials,
\cite{DicrescenzoZacks2015} for a telegraph process perturbed by a
Brownian motion, \cite{DegregorioOrsingher2011} and
\cite{GarraOrsingher2014} for certain multivariate extensions, and
\cite{Ratanov2015} for a model with jumps having some interest in
finance. Finally, since in this paper we prove results on large
deviations, we also recall \cite{Macci2016} (see also some references
cited therein) and the previous paper \cite{Macci2011}.

We also cite some references on stochastic processes with
elastic barriers: \cite{Domine1995} and \cite{Domine1996} for the
Wiener process, \cite{Giornoetal2006} for some diffusion processes
and, more recently, \cite{Jacob2012} and \cite{Jacob2013} for the
Langevin process.

Now we describe the stochastic process studied in
\cite{DicrescenzoMartinucciZacks}. It represents a random motion
of a particle on the half-line $[0,\infty)$. The particle moves up
and down in an alternating way; moreover it has velocity $1$ for
the upward periods, and it has velocity $-1$ for the downward periods.
Initially the motion proceeds upward for a positive random time
$U_1$ and, after that, the particle moves downward for a positive
random time $D_1$; moreover the motion alternates the random times
$U_2,D_2,U_3,D_3,\ldots$, where $\{U_n:n\geq 1\}$ and
$\{D_n:n\geq 1\}$ are independent sequences of i.i.d. positive
random variables. Furthermore, when the particle hits the origin,
it can be either absorbed or reflected upwards with probabilities
$\alpha$ and $1-\alpha$, respectively (here $\alpha\in(0,1)$ but,
actually, the case $\alpha=1$ is also allowed). More precisely, if
during a downward period $D_j$, say, the particle reaches the
origin and it is not absorbed, then instantaneously the particle
starts an upward period for an independent random time $U_{j+1}$.
We also remark that, here, we restrict our attention on the case
in which the random variables $\{U_n:n\geq 1\}$ and
$\{D_n:n\geq 1\}$ are exponentially distributed with parameters
$\lambda$ and $\mu$, respectively; moreover we assume that
$\lambda>\mu$ and this guarantees that
$\mathbb{E}[D_1]>\mathbb{E}[U_1]$.

In particular we are interested in the random variable
$A_x=A_x(\lambda,\mu)$, i.e. the absorption time of the particle
when it starts at $x$; see equation (2) in
\cite{DicrescenzoMartinucciZacks}. The aim is to study the
asymptotic behavior of that random variable with respect two
different scalings:
\begin{itemize}
\item \textbf{Scaling 1:} $x\to\infty$;
\item \textbf{Scaling 2:} $\mu\to\infty$, and $\lambda=\beta\mu$ for some fixed $\beta>1$ and $x>0$.
\end{itemize}
This will be done by referring to the theory of large
deviations (see e.g. \cite{DemboZeitouni} as a reference on
this topic). This theory allows to give an evaluation of
probabilities of rare events on an exponential scale. Some
preliminaries on this topic will be recalled in the next
section \ref{sec:preliminaries}.

Some of these asymptotic results concern moderate deviations;
this term is used in the literature when one has a class of
large deviation principles which fills the gap between the
convergence to a constant (typically governed by a large
deviation principle) and an asymptotic Normality result
(see Remarks \ref{rem:MD1} and \ref{rem:MD2}). Interestingly
in this paper we can also present a non-central moderate
deviation result as $\mu\to\infty$ stated in Proposition
\ref{prop:MD2-nc}; this means that we have a class of large
deviation principles that fills the gap between the convergence
to a constant and the weak convergence to a non-Gaussian limit
(see Remark \ref{rem:MD2-nc}).

In this paper we also present some numerical estimates for $\beta$
(approximate confidence intervals and point estimations)
obtained by simulations and based on an asymptotic Normality result
as $\mu\to\infty$; moreover, as far as the scaling 1 is concerned,
we also study the case in which the particle starts at some
independent random point $Y(x)$.

We conclude with the outline of the paper. We start with some
preliminaries in Section \ref{sec:preliminaries}. The results
are proved in Section \ref{sec:scaling-1} (for the scaling 1)
and in Section \ref{sec:scaling-2} (for the scaling 2). Finally,
in Section \ref{sec:numerical-estimates}, we present the
numerical estimates.

\section{Preliminaries}\label{sec:preliminaries}
We start with some preliminaries on large deviations. Then
we conclude with some details on the model studied in
\cite{DicrescenzoMartinucciZacks}; actually we recall some
preliminaries on the absorption time $A_x=A_x(\lambda,\mu)$.

\subsection{Preliminaries on large deviations}
We start with some basic definitions (see e.g. \cite{DemboZeitouni},
pages 4-5). Let $\mathcal{Z}$ be a topological space equipped with
its completed Borel $\sigma$-field. A family of $\mathcal{Z}$-valued
random variables $\{Z_r:r>0\}$ (defined on the same probability space
$(\Omega,\mathcal{F},P)$) satisfies the large deviation principle
(LDP for short) with speed function $v_r$ and rate function $I$ if:
$\lim_{r\to\infty}v_r=\infty$; the function
$I:\mathcal{Z}\to[0,\infty]$ is lower semi-continuous;
\begin{equation}\label{eq:UB-LDP-definition}
\limsup_{n\to\infty}\frac{1}{v_r}\log P(Z_r\in F)\leq-\inf_{z\in F}I(z)\ \mbox{for all closed sets}\ F;
\end{equation}
\begin{equation}\label{eq:LB-LDP-definition}
\liminf_{r\to\infty}\frac{1}{v_r}\log P(Z_r\in G)\geq-\inf_{z\in G}I(z)\ \mbox{for all open sets}\ G.
\end{equation}
A rate function $I$ is said to be \emph{good} if its level sets
$\{\{z\in\mathcal{Z}:I(z)\leq\eta\}:\eta\geq 0\}$ are compact.

Throughout this paper we prove LDPs with $\mathcal{Z}=\mathbb{R}$.
We recall the following known result for future use.

\begin{theorem}[G\"{a}rtner Ellis Theorem (on $\mathbb{R}$); see e.g. Theorem 2.3.6 in \cite{DemboZeitouni}]\label{th:GE}
Let $\{Z_r:r>0\}$ be a family of real valued random variables
(defined on the same probability space $(\Omega,\mathcal{F},P)$).
Assume that the function $\Lambda:\mathbb{R}\to(-\infty,\infty]$
defined by
$$\Lambda(s):=\lim_{r\to\infty}\frac{1}{v_r}\log\mathbb{E}\left[e^{v_rs Z_r}\right]\ (\mbox{for all}\ s\in\mathbb{R})$$
exists, and it is finite in a neighborhood of the origin $s=0$.
Moreover let
$\Lambda^*:\mathbb{R}\to[0,\infty]$ defined by
$$\Lambda^*(z):=\sup_{s\in\mathbb{R}}\{sz-\Lambda(s)\}.$$
Then: \eqref{eq:UB-LDP-definition} holds with $I=\Lambda^*$;
a weak form of \eqref{eq:LB-LDP-definition} with $I=\Lambda^*$
holds, i.e.
$$\liminf_{r\to\infty}\frac{1}{v_r}\log P(Z_r\in G)\geq-\inf_{z\in G\cap\mathcal{E}}\Lambda^*(z)\ \mbox{for all open sets}\ G$$
where $\mathcal{E}$ is the set of exposed points of $I$ (namely the
points in which $I$ is finite and strictly convex); if $\Lambda$
is essentially smooth and lower semi-continuous, then the LDP holds
with good rate function $I=\Lambda^*$.
\end{theorem}

We also recall that $\Lambda$ in the above statement is essentially
smooth (see e.g. Definition 2.3.5 in \cite{DemboZeitouni}) if
the interior of the set $\mathcal{D}_\Lambda:=\{s\in\mathbb{R}:
\Lambda(s)<\infty\}$ is non-empty, if it is differentiable
throughout the interior of that set, and if it is a steep function
(namely $|\Lambda^\prime(s)|$ tends to infinity when $s$ in the
interior of $\mathcal{D}_\Lambda$ approaches any finite point of
its boundary).

\subsection{Preliminaries on the model}
We start with a slight correction of in Proposition 9 in
\cite{DicrescenzoMartinucciZacks} for the moment generating
function of the absorption time $A_x=A_x(\lambda,\mu)$. In
what follows we consider the function
\begin{equation}\label{eq:def-Lambda}
\Lambda(s;\lambda,\mu):=\frac{1}{2}\left(\lambda-\mu-\sqrt{(\lambda+\mu-2s)^2-4\lambda\mu}\right)
\end{equation}
for $s\leq\frac{(\sqrt{\lambda}-\sqrt{\mu})^2}{2}$. Throughout
this paper we use the symbols $\Lambda^\prime(s;\lambda,\mu)$ and
$\Lambda^{\prime\prime}(s;\lambda,\mu)$ for the first and the
second derivatives of $\Lambda(s;\lambda,\mu)$ with respect to
$s$.

\begin{remark}[The moment generating function of $A_x=A_x(\lambda,\mu)$]\label{rem:correction-mgf}
Proposition 9 in \cite{DicrescenzoMartinucciZacks} provides an
expression of the moment generating function
$\mathbb{E}\left[e^{sA_x(\lambda,\mu)}\right]$ when
$s\leq\frac{(\sqrt{\lambda}-\sqrt{\mu})^2}{2}$ (it is equal to
infinity otherwise). Actually a possible further restriction on
$s$ is needed, i.e. $(1-\alpha)\mathbb{E}[e^{sC_0}]<1$, where
$\mathbb{E}[e^{sC_0}]$ is the moment generating
function of the renewal cycles $\{C_{0,i}:i\geq 1\}$
introduced in \cite{DicrescenzoMartinucciZacks}. Then, since
$\mathbb{E}\left[e^{sC_0}\right]\uparrow\sqrt{\frac{\lambda}{\mu}}$
as $s\uparrow\frac{(\sqrt{\lambda}-\sqrt{\mu})^2}{2}$, we
distinguish two cases:
\begin{itemize}
\item if $(1-\alpha)\sqrt{\frac{\lambda}{\mu}}<1$ or, equivalently, if
$\alpha>1-\sqrt{\frac{\mu}{\lambda}}$, then we recover the expression
in Proposition 9 in \cite{DicrescenzoMartinucciZacks}, i.e.
$$\mathbb{E}\left[e^{sA_x(\lambda,\mu)}\right]=\left\{\begin{array}{ll}
\frac{2\alpha\lambda e^{x\Lambda(s;\lambda,\mu)}}{2\lambda(\alpha-1)+\lambda+\mu-2s+\sqrt{(\lambda+\mu-2s)^2-4\lambda\mu}}
&\ \mbox{for}\ s\leq\frac{(\sqrt{\lambda}-\sqrt{\mu})^2}{2}\\
\infty&\ \mbox{otherwise};
\end{array}\right.$$
\item if $(1-\alpha)\sqrt{\frac{\lambda}{\mu}}\geq 1$ or, equivalently, if
$\alpha\leq 1-\sqrt{\frac{\mu}{\lambda}}$, then we have
$$\mathbb{E}\left[e^{sA_x(\lambda,\mu)}\right]=\left\{\begin{array}{ll}
\frac{2\alpha\lambda e^{x\Lambda(s;\lambda,\mu)}}{2\lambda(\alpha-1)+\lambda+\mu-2s+\sqrt{(\lambda+\mu-2s)^2-4\lambda\mu}}
&\ \mbox{for}\ s<\hat{s}(\lambda,\mu,\alpha)\\
\infty&\ \mbox{otherwise},
\end{array}\right.$$
where $\hat{s}(\lambda,\mu,\alpha):=\frac{\alpha(\lambda(1-\alpha)-\mu)}{2(1-\alpha)}
\in(0,\frac{(\sqrt{\lambda}-\sqrt{\mu})^2}{2}]$; in particular we have
$\hat{s}(\lambda,\mu,\alpha)=\frac{(\sqrt{\lambda}-\sqrt{\mu})^2}{2}$ if
$(1-\alpha)\sqrt{\frac{\lambda}{\mu}}=1$ or, equivalently, if
$\alpha=1-\sqrt{\frac{\mu}{\lambda}}$.
\end{itemize}
\end{remark}

Now we discuss some technical details on the function $\Lambda(\cdot)=
\Lambda(\cdot;\lambda,\mu)$ (in particular the concept of steepness in
the definition of essentially smooth function).

\begin{remark}[Some properties of $\Lambda(\cdot;\lambda,\mu)$]\label{rem:technical-detail-for-GET}
The function $\Lambda(\cdot)=\Lambda(\cdot;\lambda,\mu)$ plays a
crucial role in some applications of the G\"artner Ellis Theorem
in this paper. In particular, by referring to Remark
\ref{rem:correction-mgf}, it is a lower semi-continuous function
if $\alpha>1-\sqrt{\frac{\mu}{\lambda}}$, and it is an
essentially smooth function if $\alpha\geq 1-\sqrt{\frac{\mu}{\lambda}}$.
\end{remark}

In view of some results presented below, we compute the following Legendre
transforms:
$$\Lambda^*(z;\lambda,\mu):=\sup_{s\in\mathbb{R}}\{sz-\Lambda(s;\lambda,\mu)\}
=\sup_{s\leq\frac{(\sqrt{\lambda}-\sqrt{\mu})^2}{2}}\{sz-\Lambda(s;\lambda,\mu)\}$$
and, if we consider $\hat{s}(\lambda,\mu,\alpha)$ in Remark \ref{rem:correction-mgf}
for $\alpha\leq 1-\sqrt{\frac{\mu}{\lambda}}$,
$$\Lambda^*(z;\lambda,\mu,\alpha):=\sup_{s\leq\hat{s}(\lambda,\mu,\alpha)}\{sz-\Lambda(s;\lambda,\mu)\}.$$

\begin{lemma}[Computation of Legendre transforms]\label{lem:Legendre}
We have
$$\Lambda^*(z;\lambda,\mu)=\left\{\begin{array}{ll}
\frac{1}{2}\left(\sqrt{(z-1)\lambda}-\sqrt{(z+1)\mu}\right)^2&\ \mbox{if}\ z\geq 1\\
\infty&\ \mbox{otherwise}.
\end{array}\right.$$
Moreover, for $\alpha\leq 1-\sqrt{\frac{\mu}{\lambda}}$, if we set
$\tilde{z}(\lambda,\mu,\alpha):=\frac{\lambda+\mu-2\hat{s}(\lambda,\mu,\alpha)}
{\sqrt{(\lambda+\mu-2\hat{s}(\lambda,\mu,\alpha))^2-4\lambda\mu}}$ for
$\hat{s}(\lambda,\mu,\alpha)$ as in Remark \ref{rem:correction-mgf}, then
\begin{multline*}
\Lambda^*(z;\lambda,\mu,\alpha)=\left\{\begin{array}{ll}
\Lambda^*(z;\lambda,\mu)&\ \mbox{if}\ z\leq\tilde{z}(\lambda,\mu,\alpha)\\
\hat{s}(\lambda,\mu,\alpha)z-\Lambda(\hat{s}(\lambda,\mu,\alpha);\lambda,\mu)&\ \mbox{otherwise}
\end{array}\right.\\
=\left\{\begin{array}{ll}
\infty&\ \mbox{if}\ z<1\\
\frac{1}{2}\left(\sqrt{(z-1)\lambda}-\sqrt{(z+1)\mu}\right)^2&\ \mbox{if}\ 1\leq z\leq\tilde{z}(\lambda,\mu,\alpha)\\
\hat{s}(\lambda,\mu,\alpha)z-\Lambda(\hat{s}(\lambda,\mu,\alpha);\lambda,\mu)&\ \mbox{if}\ z>\tilde{z}(\lambda,\mu,\alpha).
\end{array}\right.
\end{multline*}
\end{lemma}
\begin{proof}
We start with the first statement concerning $\Lambda^*(z;\lambda,\mu)$. For $z>1$
one can check that the equation $z=\Lambda^\prime(s;\lambda,\mu)$ has solution
$s=s_z:=\frac{1}{2}\left(\lambda+\mu-2z\sqrt{\frac{\lambda\mu}{z^2-1}}\right)$,
and
$$\Lambda^*(z;\lambda,\mu)=s_zz-\Lambda(s_z;\lambda,\mu)$$
yields the desired expression; for $z\leq 1$ one can check that
$$\Lambda^*(z;\lambda,\mu)=\lim_{s\to -\infty}sz-\Lambda(s;\lambda,\mu),$$
which yields again the desired expression (one has to distinguish
the cases $z=1$ and $z<1$).

For $\Lambda^*(z;\lambda,\mu,\alpha)$ we proceed in the same way, and we omit
some details. Some computations coincide with the ones presented above but,
if $s_z$ above is larger than $\hat{s}(\lambda,\mu,\alpha)$, then the supremum
is attained at $s=\hat{s}(\lambda,\mu,\alpha)$. Moreover one can check that
$s_z>\hat{s}(\lambda,\mu,\alpha)$ if and only if $z>\tilde{z}(\lambda,\mu,\alpha)$.
The desired expression can be checked with straightforward computations.
\end{proof}

Finally, in the next Remark \ref{rem:correction-variance}, we
recall some formulas already presented in
\cite{DicrescenzoMartinucciZacks}; actually we give the
corrected expression of the variance.

\begin{remark}[A correction of a variance formula in \cite{DicrescenzoMartinucciZacks}]\label{rem:correction-variance}
Here we give the correct version of some formulas in
\cite{DicrescenzoMartinucciZacks}. More precisely we mean the
$n$-th moments of $C_x$ and $A_x$, i.e.
\begin{eqnarray*}
&& \hspace*{-0.2cm} \mathbb E[C_x^n]=\frac{\lambda}{\lambda+\mu}
{\rm e}^{\frac{x}{2} (\lambda-\mu) } \frac{2^n
n!}{(\lambda+\mu)^n}
\sum_{h=0}^{n} \left(-\frac{\lambda+\mu}{(\sqrt{\lambda}-\sqrt{\mu})^2} \right)^h
\nonumber \\
&& \hspace*{0.4cm}
\times \;
{}_{2}F_{1}\left(\frac{1+n-h}{2},\frac{2+n-h}{2};2;\frac{4\lambda
\mu}{(\lambda+\mu)^2} \right) \nonumber
\\
&& \hspace*{0.4cm} \times \sum_{j=0}^{+\infty}
\left[-\frac{(\lambda-\mu) x}{2} \right]^j \frac{1}{j!} {j/2
\choose h}  \; {}_{2}F_{1}\left(-h,-\frac{j}{2};\frac{j}{2}+1-h;
\left(\frac{\sqrt{\lambda}-\sqrt{\mu}}{\sqrt{\lambda}+\sqrt{\mu}}\right)^2
\right), \nonumber
\end{eqnarray*}
and
\begin{eqnarray*}
&& \hspace*{-0.2cm} \mathbb E[A_x^n]=2 \alpha \lambda n! {\rm
e}^{\frac{x}{2} (\lambda-\mu) }
\sum_{h=0}^{n} \left(-\frac{2}{(\sqrt{\lambda}-\sqrt{\mu})^2} \right)^h \frac{(8 \lambda (\alpha-1))^{n-h}}{(4 \lambda \alpha (\mu+ \lambda (\alpha-1)))^{n-h+1}}
\nonumber
\\
&& \hspace*{0.4cm} \times \left[ 2 \mu+2 \lambda
(\alpha-1)+\frac{2 \lambda \mu}{\lambda+\mu} \sum_{m=1}^{n-h}
\left(\frac{\alpha \mu+\alpha \lambda (\alpha-1)
}{(\alpha-1)(\lambda+\mu)} \right)^m\; \right. \nonumber
\\
&& \hspace*{0.4cm} \times \; \left.
{}_{2}F_{1}\left(\frac{m+1}{2},\frac{m+2}{2};2; \frac{4\lambda \mu}{(\lambda+\mu)^2} \right)
\right] \nonumber
\\
&& \hspace*{0.4cm} \times \sum_{j=0}^{+\infty}
\left[-\frac{(\lambda-\mu) x}{2} \right]^j \frac{1}{j!} {j/2
\choose h}\; {}_{2}F_{1}\left(-h,-\frac{j}{2};\frac{j}{2}+1-h;
\left(\frac{\sqrt{\lambda}-\sqrt{\mu}}{\sqrt{\lambda}+\sqrt{\mu}}\right)^2
\right).
\nonumber \\
&&
\end{eqnarray*}
In particular we also recall the correct expressions in
Proposition 12 in \cite{DicrescenzoMartinucciZacks} (actually only
the variances should be corrected):
$$\mathbb{E}[C_x]=\frac{2+(\lambda+\mu)x}{\lambda-\mu}\ \mbox{and}\ \mathrm{Var}[C_x]=\frac{4(\lambda+\mu+2\lambda\mu x)}{(\lambda-\mu)^3};$$
\begin{equation}\label{eq:mean-variance}
\mathbb{E}[A_x]=\frac{2+\alpha(\lambda+\mu)x}{\alpha(\lambda-\mu)}\
\mbox{and}\ \mathrm{Var}[A_x]=\frac{4(\lambda+2\lambda\mu
x\alpha^2+\mu(2\alpha-1))}{(\lambda-\mu)^3 \alpha^2}.
\end{equation}
\end{remark}

\section{Asymptotic results under the scaling 1}\label{sec:scaling-1}
We start with the standard large deviation result.

\begin{proposition}[LD as $x\to\infty$]\label{prop:LD1}
Assume that $\alpha\geq 1-\sqrt{\frac{\mu}{\lambda}}$. Then the
family $\left\{\frac{A_x}{x}:x>0\right\}$ satisfies the LDP
with speed $x$, and good rate function $I_1$ defined by
$I_1(z):=\Lambda^*(z;\lambda,\mu)$.
\end{proposition}
\begin{proof}
It is easy to check (by taking into account Remark
\ref{rem:correction-mgf}) that
$$\lim_{x\to\infty}\frac{1}{x}\log\mathbb{E}\left[e^{sA_x(\lambda,\mu)}\right]=\Lambda(s;\lambda,\mu)\ (\mbox{for all}\ s\in\mathbb{R}).$$
Then, by taking into account Remark
\ref{rem:technical-detail-for-GET}, the desired LDP holds by a
straightforward application of the G\"{a}rtner Ellis Theorem.
\end{proof}

Note that $\frac{A_x}{x}$ converges to $z_1:=\frac{\lambda+\mu}{\lambda-\mu}$
almost surely as $x\to\infty$ (this can be checked in a
standard way noting that the rate function $I_1$ uniquely
vanishes at $z_1$). Moreover
$z_1=\Lambda^\prime(0;\lambda,\mu)=\lim_{x\to\infty}\frac{\mathbb{E}[A_x]}{x}$.
Finally $z_1$ can be seen as the abscissa of the intersection
(in the $\tilde{x}\tilde{y}$ plane) of the lines $\tilde{y}=0$ and
$\tilde{y}=1+\frac{\frac{1}{\lambda}-\frac{1}{\mu}}{\frac{1}{\lambda}+\frac{1}{\mu}}\tilde{x}$.
A version of Proposition \ref{prop:LD1} concerning the case
$\alpha<1-\sqrt{\frac{\mu}{\lambda}}$ will be illustrated in Remark
\ref{rem:LD1-exposed-points-slight-extension} (case $r=1$).


Now we present the moderate deviation result. As it typically happens,
we have a class of LDPs governed by the same quadratic rate function
(i.e. $\tilde{I}_1$). Moreover this class of LDPs fills the gap
between a convergence to zero and a weak convergence to a centered
Normal distribution; see Remark \ref{rem:MD1} for some details and
comments.

\begin{proposition}[MD as $x\to\infty$]\label{prop:MD1}
For every family of positive numbers $\{\varepsilon_x:x>0\}$ such
that
\begin{equation}\label{eq:MD1-conditions}
\varepsilon_x\to 0\ \mbox{and}\ x\varepsilon_x\to\infty,
\end{equation}
the family $\left\{\frac{A_x-\mathbb{E}[A_x]}{\sqrt{x/\varepsilon_x}}:x>0\right\}$
satisfies the LDP with speed $1/\varepsilon_x$, and good rate
function $\tilde{I}_1$ defined by
$\tilde{I}_1(z):=\frac{z^2}{2\Lambda^{\prime\prime}(0;\lambda,\mu)}$,
where $\Lambda^{\prime\prime}(0;\lambda,\mu)=\frac{8\lambda\mu}{(\lambda-\mu)^3}$.
\end{proposition}
\begin{proof}
It suffices to show that
$$\lim_{x\to\infty}\frac{1}{1/\varepsilon_x}
\log\mathbb{E}\left[e^{\frac{s}{\varepsilon_x}\frac{A_x-\mathbb{E}[A_x]}{\sqrt{x/\varepsilon_x}}}\right]
=\frac{\Lambda^{\prime\prime}(0;\lambda,\mu)}{2}s^2\ (\mbox{for all}\ s\in\mathbb{R});$$
in fact the limit is a finite and differentiable function (with respect
to $s\in\mathbb{R}$) and, noting that
$$\tilde{I}_1(z)=\sup_{s\in\mathbb{R}}\left\{sz-\frac{\Lambda^{\prime\prime}(0;\lambda,\mu)}{2}s^2\right\}\ (\mbox{for all}\ z\in\mathbb{R}),$$
the desired LDP is a straightforward application of the
G\"{a}rtner Ellis Theorem.
	
We remark that
$$\frac{1}{1/\varepsilon_x}\log\mathbb{E}\left[e^{\frac{s}{\varepsilon_x}\frac{A_x-\mathbb{E}[A_x]}{\sqrt{x/\varepsilon_x}}}\right]
=\varepsilon_x\left(\log\mathbb{E}\left[e^{\frac{s}{\sqrt{x\varepsilon_x}}A_x}\right]-\frac{s}{\sqrt{x\varepsilon_x}}\mathbb{E}[A_x]\right)$$
and, since $\frac{s}{\sqrt{x\varepsilon_x}}$ is close to zero for $x$
large enough, it is easy to check (by the expressions of the moment
generating function in Remark \ref{rem:correction-mgf} and by
\eqref{eq:def-Lambda}) that
\begin{multline*}
\lim_{x\to\infty}\frac{1}{1/\varepsilon_x}\log\mathbb{E}\left[e^{\frac{s}{\varepsilon_x}\frac{A_x-\mathbb{E}[A_x]}{\sqrt{x/\varepsilon_x}}}\right]\\
=\lim_{x\to\infty}\varepsilon_x\left(x\Lambda\left(\frac{s}{\sqrt{x\varepsilon_x}};\lambda,\mu\right)-\frac{s}{\sqrt{x\varepsilon_x}}\mathbb{E}[A_x]\right)
=\lim_{x\to\infty}x\varepsilon_x\left(\Lambda\left(\frac{s}{\sqrt{x\varepsilon_x}};\lambda,\mu\right)
-\frac{s}{\sqrt{x\varepsilon_x}}\frac{\mathbb{E}[A_x]}{x}\right).
\end{multline*}
Now we take into account the Mac Laurin formula of order 2 for the
function $\Lambda(\cdot;\lambda,\mu)$, and we have
\begin{multline*}
\lim_{x\to\infty}\frac{1}{1/\varepsilon_x}\log\mathbb{E}\left[e^{\frac{s}{\varepsilon_x}\frac{A_x-\mathbb{E}[A_x]}{\sqrt{x/\varepsilon_x}}}\right]\\
=\lim_{x\to\infty}x\varepsilon_x\left(\left(\Lambda^\prime\left(0;\lambda,\mu\right)-\frac{\mathbb{E}[A_x]}{x}\right)\frac{s}{\sqrt{x\varepsilon_x}}
+\frac{\Lambda^{\prime\prime}\left(0;\lambda,\mu\right)}{2}\frac{s^2}{x\varepsilon_x}+o\left(\frac{s^2}{x\varepsilon_x}\right)\right)\\
\end{multline*}
where, by the mean value in \eqref{eq:mean-variance},
$$\Lambda^\prime\left(0;\lambda,\mu\right)-\frac{\mathbb{E}[A_x]}{x}=\frac{\lambda+\mu}{\lambda-\mu}-\frac{2+\alpha(\lambda+\mu)x}{\alpha(\lambda-\mu)x}
=-\frac{2}{\alpha(\lambda-\mu)x};$$ thus
\begin{multline*}
\lim_{x\to\infty}\frac{1}{1/\varepsilon_x}\log\mathbb{E}\left[e^{\frac{s}{\varepsilon_x}\frac{A_x-\mathbb{E}[A_x]}{\sqrt{x/\varepsilon_x}}}\right]
=\lim_{x\to\infty}-\frac{2s\sqrt{x\varepsilon_x}}{\alpha(\lambda-\mu)x}
+\frac{\Lambda^{\prime\prime}\left(0;\lambda,\mu\right)}{2}s^2
+x\varepsilon_xo\left(\frac{s^2}{x\varepsilon_x}\right)\\
=\frac{\Lambda^{\prime\prime}\left(0;\lambda,\mu\right)}{2}s^2
+\lim_{x\to\infty}-\frac{2s\sqrt{\varepsilon_x}}{\alpha(\lambda-\mu)\sqrt{x}}
+x\varepsilon_xo\left(\frac{s^2}{x\varepsilon_x}\right)=\frac{\Lambda^{\prime\prime}\left(0;\lambda,\mu\right)}{2}s^2.
\end{multline*}
\end{proof}

\begin{remark}[Typical features on MD in Proposition \ref{prop:MD1}]\label{rem:MD1}
The class of LDPs in Proposition \ref{prop:MD1} fills the gap between
the two following asymptotic regimes as $x\to\infty$:
\begin{itemize}
\item the convergence to zero of $\frac{A_x-\mathbb{E}[A_x]}{x}$ (case $\varepsilon_\mu=1/x$);
\item the weak convergence of $\frac{A_x-\mathbb{E}[A_x]}{\sqrt{x}}$ to the centered Normal
distribution with variance $\Lambda^{\prime\prime}(0;\lambda,\mu)$ (case $\varepsilon_x=1$).
\end{itemize}
In both cases one condition in \eqref{eq:MD1-conditions} holds, and the other one fails.
We also note that, by taking into account the variance expression
in \eqref{eq:mean-variance}, we have
$\Lambda^{\prime\prime}(0;\lambda,\mu)=\lim_{x\to\infty}\frac{\mathrm{Var}[A_x]}{x}$.
\end{remark}

We conclude this section by considering a generalization of Proposition \ref{prop:LD1}
with an independent random perturbation $Y(x)$ of the initial state $x$ under suitable
hypotheses collected in Condition \ref{cond:random-perturbation-initial-position} below;
this generalization will be given in Proposition \ref{prop:LD1-extended}, and it will be
followed by some remarks and comments. We start with the following slight generalization
of Proposition \ref{prop:LD1} where the initial state is modified in a deterministic way;
we recover the case in that proposition by setting $r=1$.

\begin{proposition}[Slight extension of Proposition \ref{prop:LD1}]\label{prop:LD1-slightly-extended}
Assume that $\alpha\geq 1-\sqrt{\frac{\mu}{\lambda}}$ and let $r>0$ be arbitrarily fixed.
Then the family $\left\{\frac{A_{rx}}{x}:x>0\right\}$ satisfies the LDP with speed $x$,
and good rate function $I_1(\cdot;r)$ defined by $I_1(z;r):=r\Lambda^*(z/r;\lambda,\mu)$.
\end{proposition}
\begin{proof}
It is easy to check that
$$\lim_{x\to\infty}\frac{1}{x}\log\mathbb{E}\left[e^{sA_{rx}(\lambda,\mu)}\right]=r\Lambda(s;\lambda,\mu)\ (\mbox{for all}\ s\in\mathbb{R})$$
(it is a slight modification of the analogue limit in the proof of Proposition
\ref{prop:LD1} where $r=1$).Then, by taking into account Remark
\ref{rem:technical-detail-for-GET}, the desired LDP holds by a straightforward
application of the G\"{a}rtner Ellis Theorem. In fact the governing rate
function $I_1(\cdot;r)$ is defined by
$$I_1(z;r):=\sup_{s\in\mathbb{R}}\{sz-r\Lambda(s;\lambda,\mu)\}=r\sup_{s\in\mathbb{R}}\{sz/r-\Lambda(s;\lambda,\mu)\},$$
and this coincides with the rate function in the statement of the proposition.
\end{proof}

Arguing as we did just after Proposition \ref{prop:LD1}, we can say that
$\frac{A_{rx}}{x}$ converges to $rz_1=r\frac{\lambda+\mu}{\lambda-\mu}$
almost surely as $x\to\infty$ (and the rate function $I_1(\cdot;r)$ uniquely
vanishes at $rz_1$).

\begin{remark}[Versions of Propositions \ref{prop:LD1} and \ref{prop:LD1-slightly-extended} with exposed points]\label{rem:LD1-exposed-points-slight-extension}
Here we discuss what happens when we consider the inequality
$\alpha<1-\sqrt{\frac{\mu}{\lambda}}$ in Proposition \ref{prop:LD1-slightly-extended}
(and therefore in Proposition \ref{prop:LD1} for the case $r=1$).
We have to consider some items of in the second part of Lemma
\ref{lem:Legendre} and, by the G\"artner Ellis Theorem, we have
$$\limsup_{x\to\infty}\frac{1}{x}\log P\left(\frac{A_{rx}}{x}\in F\right)
\leq-\inf_{z\in F}r\Lambda^*(z/r;\lambda,\mu,\alpha)\ \mbox{for all closed sets}\ F$$
and
$$\liminf_{x\to\infty}\frac{1}{x}\log P\left(\frac{A_{rx}}{x}\in G\right)
\geq-\inf_{z\in G\cap\mathcal{E} }r\Lambda^*(z/r;\lambda,\mu,\alpha)\ \mbox{for all open sets}\ G$$
where $\mathcal{E}=(r\tilde{z}(\lambda,\mu,\alpha),\infty)$ is the set of exposed points of
$r\Lambda^*(\cdot/r;\lambda,\mu,\alpha)$. Note that
$r\tilde{z}(\lambda,\mu,\alpha)>rz_1$, and therefore
both $r\Lambda^*(\cdot/r;\lambda,\mu,\alpha)$ and $I_1(\cdot;r)$ uniquely vanish at
$rz_1$.
\end{remark}

Now we introduce the condition on the random perturbation of the initial state.

\begin{condition}\label{cond:random-perturbation-initial-position}
Let $\{Y(x):x\geq 0\}$ be a family of nonnegative random variables
and assume that there exists the function
$\Psi_Y:\mathbb{R}\to(-\infty,\infty]$ such that
$$\Psi_Y(s):=\lim_{x\to\infty}\frac{1}{x}\log\mathbb{E}\left[e^{sY(x)}\right]$$
for all $s\in\mathbb{R}$. The function $\Psi_Y$ is nondecreasing by
construction; so we consider the set
$$\mathcal{D}_Y:=\{s\in\mathbb{R}:\Psi_Y(s)<\infty\},$$
and we assume that either $\mathcal{D}_Y=\mathbb{R}$ or, for some $\bar{s}>0$,
$\mathcal{D}_Y=(-\infty,\bar{s})$ or $\mathcal{D}_Y=(-\infty,\bar{s}]$
(note that $(-\infty,0]\subset\mathcal{D}_Y$).
\end{condition}

We remark that Condition \ref{cond:random-perturbation-initial-position} holds
when $\{Y(x):x\geq 0\}$ belongs to a wide class of nondecreasing (with respect
to $x$) L\'evy processes, also called subordinators; in this case we have
$\Psi_Y(s):=\log\mathbb{E}[e^{sY(1)}]$. For instance we recall the following
examples of infinitely divisible distributions concerning the random variable
$Y(1)$.\\

\begin{tabular}{cccccc}
Distribution&parameters&$\mathcal{D}_Y$&$\Psi_Y(s)$ for $s\in\mathcal{D}_Y$&$\Psi_Y^\prime(0)$&$\Psi_Y^{\prime\prime}(0)$\\
Poisson&$\lambda>0$&$\mathbb{R}$&$\lambda(e^s-1)$&$\lambda$&$\lambda$\\
Gamma&$\lambda,\theta>0$&$(-\infty,\bar{s}),\ \bar{s}=\theta$&$\gamma\log\frac{\theta}{\theta-s}$&$\frac{\lambda}{\theta}$&$\frac{\lambda}{\theta^2}$\\
Inverse Gaussian&$\xi>0$&$(-\infty,\bar{s}],\ \bar{s}=\frac{\xi^2}{2}$&$\xi-\sqrt{\xi^2-2s}$&$\xi^{-1}$&$\xi^{-3}$
\end{tabular}
\ \\
\ \\
So now we are ready to state the main generalization of Proposition \ref{prop:LD1}.

\begin{proposition}[Extension of Proposition \ref{prop:LD1}]\label{prop:LD1-extended}
Assume that $\alpha\geq 1-\sqrt{\frac{\mu}{\lambda}}$ and that a process $\{Y(x):x\geq 0\}$,
independent of $\{A_x:x\geq 0\}$, satisfies Condition
\ref{cond:random-perturbation-initial-position}. Moreover let $\Lambda_Y$ be the
function defined by
$$\Lambda_Y(s):=\left\{\begin{array}{ll}
\Psi_Y(\Lambda(s;\lambda,\mu))&\ \mbox{for}\ \Lambda(s;\lambda,\mu)\in\mathcal{D}_Y\ \mbox{and}\ s\leq\frac{(\sqrt{\lambda}-\sqrt{\mu})^2}{2}\\
\infty&\ \mbox{otherwise},
\end{array}\right.$$
and assume that it is essentially smooth. Then the family
$\left\{\frac{A_{Y(x)}}{x}:x>0\right\}$ satisfies the LDP with speed $x$, and good rate
function $\Lambda_Y^*$ defined by $\Lambda_Y^*(z):=\sup_{s\in\mathbb{R}}\{sz-\Lambda_Y(s)\}$.
\end{proposition}
\begin{proof}
We compute the moment generating function of $A_{Y(x)}$ by considering the well-known
equality $\mathbb{E}\left[e^{sA_{Y(x)}(\lambda,\mu)}\right]
=\mathbb{E}\left[\left.\mathbb{E}\left[e^{sA_r(\lambda,\mu)}\right]\right|_{r=Y(x)}\right]$.
Moreover, by the expression of the moment generating function in Remark
\ref{rem:correction-mgf}, we get
$$\mathbb{E}\left[e^{sA_{Y(x)}(\lambda,\mu)}\right]=\left\{\begin{array}{ll}
\frac{2\alpha\lambda\mathbb{E}\left[e^{Y(x)\Lambda(s;\lambda,\mu)}\right]}{2\lambda(\alpha-1)+\lambda+\mu-2s+\sqrt{(\lambda+\mu-2s)^2-4\lambda\mu}}
&\ \mbox{for}\ s\leq\frac{(\sqrt{\lambda}-\sqrt{\mu})^2}{2}\\
\infty&\ \mbox{otherwise}.
\end{array}\right.$$
So, by Condition \ref{cond:random-perturbation-initial-position}, we get
$$\lim_{x\to\infty}\frac{1}{x}\log\mathbb{E}\left[e^{sA_{Y(x)}(\lambda,\mu)}\right]=\Lambda_Y(s)\ (\mbox{for all}\ s\in\mathbb{R}),$$
where $\Lambda_Y$ is the function in the statement of the proposition. In conclusion,
since $\Lambda_Y$ is an essentially smooth function, the desired LDP holds by a
straightforward application of the G\"{a}rtner Ellis Theorem.
\end{proof}

Now we present some remarks and comments on Proposition \ref{prop:LD1-extended}.
In what follows we assume that the function $\Psi_Y$ is differentiable in
the interior of $\mathcal{D}_Y$.

\begin{remark}[Extension of some parts in the proof of Lemma \ref{lem:Legendre}]\label{rem:extension-of-something-in-Lemma}
We consider $\Psi_Y^\prime(-\infty):=\lim_{s\to -\infty}\Psi_Y^\prime(s)$ (this limit is
well-defined because $\Psi_Y^\prime$ is monotonic by the convexity of $\Psi_Y$). Then,
for $z>\Psi_Y^\prime(-\infty)$, one can check that the equation $z=\Lambda_Y^\prime(s)$
has solution $s=\tilde{s}_z$, and we have
$$\Lambda_Y^*(z)=z\tilde{s}_z-\Lambda_Y(\tilde{s}_z).$$
On the other hand, for $z\leq\Psi_Y^\prime(-\infty)$ one can check that
$$\Lambda_Y^*(z)=\lim_{s\to -\infty}sz-\Lambda_Y(s),$$
which is finite for $z=\Psi_Y^\prime(-\infty)$ and infinite for $z<\Psi_Y^\prime(-\infty)$.
\end{remark}

\begin{remark}[On the essential smoothness of $\Lambda_Y$]\label{rem:on-essential-smoothness}
In general, if $s$ belongs to the interior of the set where $\Lambda_Y(s)<\infty$, we have
$$\Lambda_Y^\prime(s)=\Psi_Y^\prime(\Lambda(s;\lambda,\mu))\Lambda^\prime(s;\lambda,\mu).$$
Then, if we refer to Condition \ref{cond:random-perturbation-initial-position}, we have two cases.
\begin{enumerate}
\item If $\mathcal{D}_Y=\mathbb{R}$, then we have to check that $\Lambda_Y^\prime(s)\uparrow\infty$
as $s\uparrow\frac{(\sqrt{\lambda}-\sqrt{\mu})^2}{2}$. This statement is true because
$\Lambda^\prime(s;\lambda,\mu)\uparrow\infty$ and $\Psi_Y^\prime(\Lambda(s;\lambda,\mu))$ tends to
a positive limit.
\item If we have $\mathcal{D}_Y=(-\infty,\bar{s})$ or $\mathcal{D}_Y=(-\infty,\bar{s}]$ for some
$\bar{s}\in(0,\infty)$, then we take
$s_0:=\Lambda^{-1}(\bar{s};\lambda,\mu)\wedge\frac{(\sqrt{\lambda}-\sqrt{\mu})^2}{2}$ and we have
to check that $\Lambda_Y^\prime(s)\uparrow\infty$ as $s\uparrow s_0$. If
$s_0=\frac{(\sqrt{\lambda}-\sqrt{\mu})^2}{2}$, then we can conclude following the lines of the
previous case ($\mathcal{D}_Y=\mathbb{R}$). If $s_0=\Lambda^{-1}(\bar{s};\lambda,\mu)$,
we also require the condition $\Psi_Y^\prime(s)\uparrow\infty$ as $s\uparrow\bar{s}$, and then we
have $\Lambda_Y^\prime(s)\uparrow\infty$ because $\Psi_Y^\prime(\Lambda(s;\lambda,\mu))\uparrow\infty$
and $\Lambda^\prime(s;\lambda,\mu)$ tends to a positive limit.
\end{enumerate}
\end{remark}

We continue with some further comments and, from now on, we assume that
$\Psi_Y^\prime(0)>0$; note that this condition holds for the examples
tabulated above. Moreover we assume to have the hypotheses of Propositions
\ref{prop:LD1-slightly-extended} and \ref{prop:LD1-extended} that
guarantee the validity of the LDPs stated in those propositions. It is known
that $\Lambda_Y^*(z)=0$ if and only if
$z=\hat{z}:=\Lambda_Y^\prime(0)=\Psi_Y^\prime(0)\Lambda^\prime(0;\lambda,\mu)$,
and $I_1(z;r)=0$ if and only if $z=z_r^*:=r\Lambda^\prime(0;\lambda,\mu)$. So,
if we take $r=\Psi_Y^\prime(0)$, we have $\hat{z}=z_r^*$, both rate functions
$\Lambda_Y^*$ and $I_1(\cdot;r)$ uniquely vanish at $\hat{z}$, and therefore
both $\frac{A_{Y(x)}}{x}$ and $\frac{A_{rx}}{x}$ converge to same limit
$\hat{z}$ (as $x\to\infty$). Thus, in this case, it is interesting to find
inequalities between rate functions (when $z$ belongs to a neighborhood of
$\hat{z}$, except $z=\hat{z}$) to say that we have a faster
convergence in the case governed by the (locally) larger rate function.

We start noting that $\Psi_Y(s)\geq\Psi_Y^\prime(0) s$
by the convexity of $\Psi_Y$ and by $\Psi_Y(0)=0$; thus we obtain
$$\Lambda_Y(s)\geq\Psi_Y^\prime(0)\Lambda(s;\lambda,\mu)\ \mbox{for all}\ s\in\mathbb{R}$$
(in fact, if $s>\frac{(\sqrt{\lambda}-\sqrt{\mu})^2}{2}$, we
have $\Lambda_Y(s)=\Psi_Y^\prime(0)\Lambda(s;\lambda,\mu)=\infty$).
So we easily obtain the following inequality between rate functions:
$$\Lambda_Y^*(z)=\sup_{s\in\mathbb{R}}\{sz-\Lambda_Y(s)\}\leq
\sup_{s\in\mathbb{R}}\{sz-\Psi_Y^\prime(0)\Lambda(s;\lambda,\mu)\}=I_1(z;\Psi_Y^\prime(0)).$$
In conclusion the rate function which governs the LDP of
$\left\{\frac{A_{\Psi_Y^\prime(0)x}}{x}:x>0\right\}$ cannot be smaller
than the one for the LDP of $\left\{\frac{A_{Y(x)}}{x}:x>0\right\}$;
this is not surprising because we expect to have a faster convergence
(to $\hat{z}$ as $x\to\infty$) when the perturbation of the initial
position is deterministic.

We also remark that, under suitable conditions (for instance if $\Psi_Y$
is strictly convex, as happens for the tabulated examples above), we
have the strict inequality $\Lambda_Y^*(z)<I_1(z;\Psi_Y^\prime(0))$ except
for the cases in which both $\Lambda_Y^*(z)$ and $I_1(z;\Psi_Y^\prime(0))$
are equal to zero (i.e. if $z=\hat{z}$) or to infinity (i.e. if
$z<\Psi_Y^\prime(-\infty)$).

As a final comment we also briefly discuss the comparison between
the second derivatives of the rate functions at $z=\hat{z}$; indeed
a larger second derivative corresponds to a locally larger rate
function in a neighborhood of $\hat{z}$, except $z=\hat{z}$. We have
$$(\Lambda_Y^*)^{\prime\prime}(\Lambda_Y^\prime(0))=\frac{1}{\Lambda_Y^{\prime\prime}(0)}
=\frac{1}{\Psi_Y^{\prime\prime}(0)(\Lambda^\prime(0;\lambda,\mu))^2+\Psi_Y^\prime(0)\Lambda^{\prime\prime}(0;\lambda,\mu)}$$
and
$$I_1^{\prime\prime}(r\Lambda^\prime(0;\lambda,\mu);r)=\frac{1}{r\Lambda^{\prime\prime}(0;\lambda,\mu)};$$
thus, if we set $r=\Psi_Y^\prime(0)$ in the last equalities, we get
$$(\Lambda_Y^*)^{\prime\prime}(\hat{z})\leq I_1^{\prime\prime}(\hat{z};\Psi_Y^\prime(0))$$
by the convexity of the function $\Psi_Y$ which yields $\Psi_Y^{\prime\prime}(0)\geq 0$.
Actually in several common cases we have the strict inequality
$(\Lambda_Y^*)^{\prime\prime}(\hat{z})<I_1^{\prime\prime}(\hat{z};\Psi_Y^\prime(0))$
because $\Lambda^\prime(0;\lambda,\mu)>0$ and, as happens for the tabulated examples above,
$\Psi_Y^{\prime\prime}(0)>0$.

\section{Asymptotic results under the scaling 2}\label{sec:scaling-2}
Throughout this section we set $\lambda=\beta\mu$ for some
$\beta>1$ and $x>0$. We start with the standard large deviation result.

\begin{proposition}[LD as $\mu\to\infty$]\label{prop:LD2}
Assume that $\alpha\geq 1-\sqrt{\frac{1}{\beta}}$. Then the	family
$\{A_x(\beta\mu,\mu):\mu>0\}$ satisfies the LDP with speed $\mu$,
and good rate function $I_2$ defined by
$I_2(z):=x\Lambda^*(z/x;\beta,1)$.
\end{proposition}
\begin{proof}
It is easy to check (by taking into account Remark
\ref{rem:correction-mgf}) that
$$\lim_{\mu\to\infty}\frac{1}{\mu}\log\mathbb{E}\left[e^{\mu sA_x(\beta\mu,\mu)}\right]=x\Lambda(s;\beta,1)\ (\mbox{for all}\ s\in\mathbb{R}).$$
Then, by taking into account Remark
\ref{rem:technical-detail-for-GET}, the desired LDP holds by a
straightforward application of the G\"{a}rtner Ellis Theorem. In
fact the governing rate function $I_2$ is defined by
$$I_2(z):=\sup_{s\in\mathbb{R}}\{sz-x\Lambda(s;\beta,1)\}=x\sup_{s\in\mathbb{R}}\{sz/x-\Lambda(s;\beta,1)\},$$
and this coincides with the rate function in the statement of the proposition.
\end{proof}

Note that $A_x(\beta\mu,\mu)$ converges to $z_2:=x\frac{\beta+1}{\beta-1}$
almost surely as $\mu\to\infty$ (in fact the rate function $I_2$ uniquely
vanishes at $z_2$). Moreover
$z_2=x\Lambda^\prime(0;\beta,1)=\lim_{\mu\to\infty}\mathbb{E}[A_x(\beta\mu,\mu)]$.
Finally $z_2$ can be seen as the abscissa of the intersection
(in the $\tilde{x}\tilde{y}$ plane) of the lines $\tilde{y}=0$ and
$\tilde{y}=x+\frac{\frac{1}{\beta\mu}-\frac{1}{\mu}}{\frac{1}{\beta\mu}+\frac{1}{\mu}}\tilde{x}$.

\begin{remark}[A version of Proposition \ref{prop:LD2} with exposed points]\label{rem:LD2-exposed-points}
Here we discuss what happens when we consider the inequality
$\alpha<1-\sqrt{\frac{1}{\beta}}$ in Proposition \ref{prop:LD2}.
In this case we still have to consider some items of in the second part of
Lemma \ref{lem:Legendre} (as in Remark \ref{rem:LD1-exposed-points-slight-extension})
and, by the G\"artner Ellis Theorem, we have
$$\limsup_{\mu\to\infty}\frac{1}{\mu}\log P\left(A_x(\beta\mu,\mu)\in F\right)
\leq-\inf_{z\in F}x\Lambda^*(z/x;\beta,1,\alpha)\ \mbox{for all closed sets}\ F$$
and
$$\liminf_{\mu\to\infty}\frac{1}{\mu}\log P\left(A_x(\beta\mu,\mu)\in G\right)
\geq-\inf_{z\in G\cap\mathcal{E}}x\Lambda^*(z/x;\beta,1,\alpha)\ \mbox{for all open sets}\ G$$
where $\mathcal{E}=(x\tilde{z}(\beta,1,\alpha),\infty)$ is the set of exposed points of
$x\Lambda^*(\cdot/x;\beta,1,\alpha)$. Note that
$x\tilde{z}(\beta,1,\alpha)>z_2$, and therefore
both $x\Lambda^*(\cdot/x;\beta,1,\alpha)$ and $I_2$ uniquely vanish at $z_2$.
\end{remark}


Now we present the moderate deviation result. As it typically happens,
we have a class of LDPs governed by the same quadratic rate function
(i.e. $\tilde{I}_2$). Moreover this class of LDPs fills the gap
between a convergence to zero and a weak convergence to a centered
Normal distribution; see Remark \ref{rem:MD2} for some details and
comments.

\begin{proposition}[MD as $\mu\to\infty$]\label{prop:MD2}
For every family of positive numbers $\{\varepsilon_\mu:\mu>0\}$
such that
\begin{equation}\label{eq:MD2-conditions}
\varepsilon_\mu\to 0\ \mbox{and}\ \mu\varepsilon_\mu\to\infty,
\end{equation}
the family $\left\{\sqrt{\mu\varepsilon_\mu}(A_x(\beta\mu,\mu)-\mathbb{E}[A_x(\beta\mu,\mu)]):\mu>0\right\}$
satisfies the LDP with speed $1/\varepsilon_\mu$, and good rate
function $\tilde{I}_2$ defined by
$\tilde{I}_2(z):=\frac{z^2}{2x\Lambda^{\prime\prime}(0;\beta,1)}$,
where $\Lambda^{\prime\prime}(0;\beta,1)=\frac{8\beta}{(\beta-1)^3}$.
\end{proposition}
\begin{proof}
It suffices to show that
$$\lim_{\mu\to\infty}\frac{1}{1/\varepsilon_\mu}
\log\mathbb{E}\left[e^{\frac{s}{\varepsilon_\mu}\sqrt{\mu\varepsilon_\mu}
(A_x(\beta\mu,\mu)-\mathbb{E}[A_x(\beta\mu,\mu)])}\right]
=\frac{x\Lambda^{\prime\prime}(0;\beta,1)}{2}s^2\ (\mbox{for all}\ s\in\mathbb{R});$$
in fact the limit is a finite and differentiable function (with
respect to $s\in\mathbb{R}$) and, noting that
$$\tilde{I}_2(z)=\sup_{s\in\mathbb{R}}\left\{sz-\frac{x\Lambda^{\prime\prime}(0;\beta,1)}{2}s^2\right\}\ (\mbox{for all}\ z\in\mathbb{R}),$$
the desired LDP is a straightforward application of the
G\"{a}rtner Ellis Theorem.
	
We remark that
$$\frac{1}{1/\varepsilon_\mu}\log\mathbb{E}\left[e^{\frac{s}{\varepsilon_\mu}\sqrt{\mu\varepsilon_\mu}(A_x(\beta\mu,\mu)-\mathbb{E}[A_x(\beta\mu,\mu)])}\right]
=\varepsilon_\mu\left(\log\mathbb{E}\left[e^{s\frac{\sqrt{\mu}}{\sqrt{\varepsilon_\mu}}A_x(\beta\mu,\mu)}\right]
-\frac{s\sqrt{\mu}}{\sqrt{\varepsilon_\mu}}\mathbb{E}[A_x(\beta\mu,\mu)]\right);$$
moreover
$$\log\mathbb{E}\left[e^{s\frac{\sqrt{\mu}}{\sqrt{\varepsilon_\mu}}A_x(\beta\mu,\mu)}\right]
=\log\mathbb{E}\left[e^{\frac{s\mu}{\sqrt{\mu\varepsilon_\mu}}A_x(\beta\mu,\mu)}\right]
=\log\mathbb{E}\left[e^{\frac{s}{\sqrt{\mu\varepsilon_\mu}}A_{x\mu}(\beta,1)}\right]$$
where the last equality holds by the expressions of the moment generating function in Remark
\ref{rem:correction-mgf} and by \eqref{eq:def-Lambda}; then, since
$\frac{s}{\sqrt{\mu\varepsilon_\mu}}$ is close to zero for $\mu$ large enough, it is easy to
check that
\begin{multline*}
\lim_{\mu\to\infty}\frac{1}{1/\varepsilon_\mu}
\log\mathbb{E}\left[e^{\frac{s}{\varepsilon_\mu}\sqrt{\mu\varepsilon_\mu}(A_x(\beta\mu,\mu)-\mathbb{E}[A_x(\beta\mu,\mu)])}\right]\\
=\lim_{\mu\to\infty}\varepsilon_\mu\left(x\mu\Lambda\left(\frac{s}{\sqrt{\mu\varepsilon_\mu}};\beta,1\right)
-\frac{s\mu}{\sqrt{\mu\varepsilon_\mu}}\mathbb{E}[A_x(\beta\mu,\mu)]\right)\\
=\lim_{\mu\to\infty}\mu\varepsilon_\mu\left(x\Lambda\left(\frac{s}{\sqrt{\mu\varepsilon_\mu}};\beta,1\right)
-\frac{s}{\sqrt{\mu\varepsilon_\mu}}\mathbb{E}[A_x(\beta\mu,\mu)]\right).
\end{multline*}
Now we take into account the Mac Laurin formula of order 2 for the
function $\Lambda(\cdot;\beta,1)$, and we have
\begin{multline*}
\lim_{\mu\to\infty}\frac{1}{1/\varepsilon_\mu}
\log\mathbb{E}\left[e^{\frac{s}{\varepsilon_\mu}\sqrt{\mu\varepsilon_\mu}(A_x(\beta\mu,\mu)-\mathbb{E}[A_x(\beta\mu,\mu)])}\right]\\
=\lim_{\mu\to\infty}\mu\varepsilon_\mu\left(\left(x\Lambda^\prime\left(0;\beta,1\right)-\mathbb{E}[A_x(\beta\mu,\mu)]\right)\frac{s}{\sqrt{\mu\varepsilon_\mu}}
+\frac{x\Lambda^{\prime\prime}\left(0;\beta,1\right)}{2}\frac{s^2}{\mu\varepsilon_\mu}+o\left(\frac{s^2}{\mu\varepsilon_\mu}\right)\right)\\
\end{multline*}
where, by the mean value in \eqref{eq:mean-variance},
$$x\Lambda^\prime\left(0;\beta,1\right)-\mathbb{E}[A_x(\beta\mu,\mu)]=x\frac{\beta+1}{\beta-1}-\frac{2+\alpha(\beta+1)\mu x}{\alpha(\beta-1)\mu}
=-\frac{2}{\alpha(\beta-1)\mu};$$ thus
\begin{multline*}
\lim_{\mu\to\infty}\frac{1}{1/\varepsilon_\mu}
\log\mathbb{E}\left[e^{\frac{s}{\varepsilon_\mu}\sqrt{\mu\varepsilon_\mu}(A_x(\beta\mu,\mu)-\mathbb{E}[A_x(\beta\mu,\mu)])}\right]\\
=\frac{x\Lambda^{\prime\prime}\left(0;\beta,1\right)}{2}s^2
+\lim_{\mu\to\infty}-\frac{2\sqrt{\varepsilon_\mu}}{\alpha(\beta-1)\sqrt{\mu}}s+\mu\varepsilon_\mu
o\left(\frac{s^2}{\mu\varepsilon_\mu}\right)=\frac{x\Lambda^{\prime\prime}\left(0;\beta,1\right)}{2}s^2.
\end{multline*}
\end{proof}

\begin{remark}[Typical features on MD in Proposition \ref{prop:MD2}]\label{rem:MD2}
The class of LDPs in Proposition \ref{prop:MD2} fills the gap between
the two following asymptotic regimes as $\mu\to\infty$:
\begin{itemize}
\item the convergence to zero of $A_x(\beta\mu,\mu)-\mathbb{E}[A_x(\beta\mu,\mu)]$ (case $\varepsilon_\mu=1/\mu$);
\item the weak convergence of $\sqrt{\mu}(A_x(\beta\mu,\mu)-\mathbb{E}[A_x(\beta\mu,\mu)])$ to the centered Normal
distribution with variance $x\Lambda^{\prime\prime}(0;\beta,1)$ (case $\varepsilon_\mu=1$).
\end{itemize}
In both cases one condition in \eqref{eq:MD2-conditions} holds, and the other one fails.
We also note that, by taking into account the variance expression
in \eqref{eq:mean-variance}, we have
$x\Lambda^{\prime\prime}(0;\beta,1)=\lim_{\mu\to\infty}\mu\mathrm{Var}[A_x(\beta\mu,\mu)]$.
\end{remark}

We conclude this section with another moderate deviation result, which will
be stated in Proposition \ref{prop:MD2-nc}. Namely we mean a class of LDPs
that fills the gap between two asymptotic regimes, as $\mu\to\infty$,
as in Proposition \ref{prop:MD2}; more precisely the convergence to a
constant, and the weak convergence to a suitable non degenerate limit law
(this will be explained in Remark \ref{rem:MD2-nc} below). In some sense we
have a non-central moderate deviation result because the limit law is
non-Gaussian; actually, as shown in the next Lemma \ref{lem:weak-convergence},
we deal with a family of equally distributed random variables and therefore
the weak convergence trivially holds.

\begin{lemma}[A weak convergence result as $\mu\to\infty$]\label{lem:weak-convergence}
The random variables $\left\{\mu A_{x/\mu}(\beta\mu,\mu):\mu>0\right\}$ are equally
distributed.
\end{lemma}
\begin{proof}
The result can be easily proved by taking the moment generating function
of the involved random variables, and by referring to the formulas presented
in Remark \ref{rem:correction-mgf}. One can easily check (we omit the details)
that, under every condition on $\alpha$ stated in Remark \ref{rem:correction-mgf},
we have the same moment generating function for every random variables of the
family $\left\{\mu A_{x/\mu}(\beta\mu,\mu):\mu>0\right\}$ (in fact it does not
depend on $\mu$).
\end{proof}

Now we can prove the non-central moderate deviation result.

\begin{proposition}[Non-central MD as $\mu\to\infty$]\label{prop:MD2-nc}
Assume that $\alpha\geq 1-\sqrt{\frac{1}{\beta}}$. Then, for every family of
positive numbers $\{\varepsilon_\mu:\mu>0\}$ such that \eqref{eq:MD2-conditions}
holds, the family
$\left\{\mu\varepsilon_\mu A_{x/(\mu\varepsilon_\mu)}(\beta\mu,\mu):\mu>0\right\}$
satisfies the LDP with speed $1/\varepsilon_\mu$, and good rate
function $I_2$ (presented in Proposition \ref{prop:LD2}).
\end{proposition}
\begin{proof}
We want to apply the G\"{a}rtner Ellis Theorem. So we have
\begin{multline*}
\frac{1}{1/\varepsilon_\mu}\log\mathbb{E}\left[e^{\frac{s}{\varepsilon_\mu}\mu\varepsilon_\mu A_{x/(\mu\varepsilon_\mu)}(\beta\mu,\mu)}\right]
=\varepsilon_\mu\log\mathbb{E}\left[e^{s\mu A_{x/(\mu\varepsilon_\mu)}(\beta\mu,\mu)}\right]\\
=\left\{\begin{array}{ll}
\varepsilon_\mu\log\frac{2\alpha\beta\mu e^{x/(\mu\varepsilon_\mu)\Lambda(s\mu;\beta\mu,\mu)}}{2\beta\mu(\alpha-1)+\beta\mu+\mu-2s\mu+\sqrt{(\beta\mu+\mu-2s\mu)^2-4\beta\mu^2}}
&\ \mbox{for}\ s\mu\leq\frac{(\sqrt{\beta\mu}-\sqrt{\mu})^2}{2}\\
\infty&\ \mbox{otherwise}
\end{array}\right.\\
=\left\{\begin{array}{ll}
\varepsilon_\mu\log\frac{2\alpha\beta e^{x\Lambda(s;\beta,1)/\varepsilon_\mu}}{2\beta(\alpha-1)+\beta+1-2s+\sqrt{(\beta+1-2s)^2-4\beta}}
&\ \mbox{for}\ s\leq\frac{(\sqrt{\beta}-1)^2}{2}\\
\infty&\ \mbox{otherwise};
\end{array}\right.
\end{multline*}
then
$$\lim_{\mu\to\infty}\frac{1}{1/\varepsilon_\mu}\log\mathbb{E}\left[e^{\frac{s}{\varepsilon_\mu}\mu\varepsilon_\mu
A_{x/(\mu\varepsilon_\mu)}(\beta\mu,\mu)}\right]=x\Lambda(s;\beta,1)\ (\mbox{for all}\ s\in\mathbb{R}),$$
and, by Remark \ref{rem:technical-detail-for-GET}, the desired LDP is a
straightforward application of the G\"{a}rtner Ellis Theorem.
\end{proof}

In the next remark we follow the same lines of Remarks
\ref{rem:LD1-exposed-points-slight-extension} and
\ref{rem:LD2-exposed-points}.

\begin{remark}[A version of Proposition \ref{prop:MD2-nc} with exposed points]\label{rem:MD2-nc-exposed-points}
Here we discuss what happens when we consider the inequality
$\alpha<1-\sqrt{\frac{1}{\beta}}$ in Proposition \ref{prop:MD2-nc}.
In this case we still have to consider some items of in the second part of
Lemma \ref{lem:Legendre} (as in Remarks \ref{rem:LD1-exposed-points-slight-extension}
and \ref{rem:LD2-exposed-points}) and, by the G\"artner Ellis Theorem, we have
$$\limsup_{\mu\to\infty}\frac{1}{1/\varepsilon_\mu}\log P\left(\mu\varepsilon_\mu A_{x/(\mu\varepsilon_\mu)}(\beta\mu,\mu)\in F\right)
\leq-\inf_{z\in F}x\Lambda^*(z/x;\beta,1,\alpha)\ \mbox{for all closed sets}\ F$$
and
$$\liminf_{\mu\to\infty}\frac{1}{1/\varepsilon_\mu}\log P\left(\mu\varepsilon_\mu A_{x/(\mu\varepsilon_\mu)}(\beta\mu,\mu)\in G\right)
\geq-\inf_{z\in G\cap\mathcal{E}}x\Lambda^*(z/x;\beta,1,\alpha)\ \mbox{for all open sets}\ G$$
where $\mathcal{E}=(x\tilde{z}(\beta,1,\alpha),\infty)$ is the set of exposed points of
$x\Lambda^*(\cdot/x;\beta,1,\alpha)$. Note that
$x\tilde{z}(\beta,1,\alpha)>z_2$, and therefore
both $x\Lambda^*(\cdot/x;\beta,1,\alpha)$ and $I_2$ uniquely vanish at $z_2$.
\end{remark}

We conclude with the analogue of Remark \ref{rem:MD2}, where we also give some
comments on the limit of the scaled variance.

\begin{remark}[The analogue of Remark \ref{rem:MD2}]\label{rem:MD2-nc}
The class of LDPs in Proposition \ref{prop:MD2-nc} fills the gap between
the two following asymptotic regimes as $\mu\to\infty$:
\begin{itemize}
\item the convergence of $A_x(\beta\mu,\mu)$ to $x\Lambda^\prime(0;\beta,1)$
(case $\varepsilon_\mu=1/\mu$), which follows from the LDP in Proposition
\ref{prop:LD2};
\item the weak convergence of $\mu A_{x/\mu}(\beta\mu,\mu)$ to $A_x(\beta,1)$ (case
$\varepsilon_\mu=1$) proved in Lemma \ref{lem:weak-convergence}.
\end{itemize}
In both cases one condition in \eqref{eq:MD2-conditions} holds, and the other one fails.
We can also provide the following limit for the scaled variance (where we take into
account the variance expression in \eqref{eq:mean-variance}):
$$\lim_{\mu\to\infty}\frac{1}{\varepsilon_\mu}
\mathrm{Var}\left[\mu\varepsilon_\mu A_{x/(\mu\varepsilon_\mu)}(\beta\mu,\mu)\right]=x\Lambda^{\prime\prime}(0;\beta,1).$$
Thus the variance of the equally distributed random variables in Lemma
\ref{lem:weak-convergence} (and therefore the variance of the weak limit
$A_x(\beta,1)$) can be expressed as
$$\mathrm{Var}[A_x(\beta,1)]=\frac{4(\beta+2\beta x\alpha^2+2\alpha-1)}{(\beta-1)^3\alpha^2}
=\Delta(\alpha,\beta)+x\Lambda^{\prime\prime}(0;\beta,1),$$
where $x\Lambda^{\prime\prime}(0;\beta,1)$ is the limit value obtained above, and
$\Delta(\alpha,\beta):=\frac{4(\beta+2\alpha-1)}{(\beta-1)^3\alpha^2}>0$; moreover
$\Delta(\alpha,\beta)$ tends to zero as $\beta\to\infty$.
\end{remark}

\section{Numerical estimates by simulations}\label{sec:numerical-estimates}
In this section we refer to the asymptotic Normality result under
the scaling 2 stated in Remark \ref{rem:MD2}. We present
numerical values obtained by simulations to estimate $\beta$;
actually we assume that $\beta>\beta_0$ for some known $\beta_0>1$.
In the final part we also present some figures concerning sample
paths for some $\beta>1$.

We denote the standard Normal distribution by $\Phi$. Then, for
every $\delta>0$, we have
$$\lim_{\mu\to\infty}P\left(A_x(\beta\mu,\mu)-\frac{\delta}{\sqrt{\mu}}\leq
\mathbb{E}[A_x(\beta\mu,\mu)]\leq A_x(\beta\mu,\mu)+\frac{\delta}{\sqrt{\mu}}\right)=
2\Phi\left(\frac{\delta}{\sqrt{8\beta x/(\beta-1)^3}}\right)-1;$$
so, if we choose $\delta=\sqrt{\frac{8\beta x}{(\beta-1)^3}}\Phi^{-1}\left(\frac{1+\ell}{2}\right)$
for some $\ell\in(0,1)$, the above limit is equal to $\ell$. Thus we
can consider the following approximated confidence interval for
$\mathbb{E}[A_x(\beta\mu,\mu)]$ at the level $\ell$, when $\mu$ is
large:
$$A_x(\beta\mu,\mu)\pm\sup_{\beta>\beta_0}\sqrt{\frac{8\beta x}{(\beta-1)^3}}\frac{\Phi^{-1}\left(\frac{1+\ell}{2}\right)}{\sqrt{\mu}}.$$
We already remarked just after Proposition \ref{prop:LD2} that
$$\lim_{\mu\to\infty}\mathbb{E}[A_x(\beta\mu,\mu)]=x\Lambda^\prime(0;\beta,1)=x\frac{\beta+1}{\beta-1};$$
thus, for $\mu$ large enough ($\mu>\mu_0$, say) the approximation
$\mathbb{E}[A_x(\beta\mu,\mu)]\approx x\frac{\beta+1}{\beta-1}$ can be
adopted. Moreover, since $\sup_{\beta>\beta_0}\sqrt{\frac{8\beta x}{(\beta-1)^3}}=
\sqrt{\frac{8\beta_0x}{(\beta_0-1)^3}}$, the above approximated confidence
interval can be specified as follows:
$$A_x(\beta\mu,\mu)\pm\sqrt{\frac{8\beta_0x}{(\beta_0-1)^3}}\frac{\Phi^{-1}\left(\frac{1+\ell}{2}\right)}{\sqrt{\mu}}.$$

Then we can obtain numerical values for this confidence interval by
performing simulations of $A_x(\beta\mu,\mu)$. Specifically, the
validate simulations of $A_x(\beta\mu,\mu)$ are those performed for
selected values of $\beta$, i.e. for chosen values $\beta=\beta_*>\beta_0>1$,
for which the fraction of sample paths such that
$$x\frac{\beta_*+1}{\beta_*-1}\in\left(\overline{A}_x(\beta_*\mu,\mu)-\sqrt{\frac{8\beta_0x}{(\beta_0-1)^3}}
\frac{\Phi^{-1}\left(\frac{1+\ell}{2}\right)}{\sqrt{\mu}},
\overline{A}_x(\beta_*\mu,\mu)+\sqrt{\frac{8\beta_0x}{(\beta_0-1)^3}}
\frac{\Phi^{-1}\left(\frac{1+\ell}{2}\right)}{\sqrt{\mu}}\right),$$
where $\overline{A}_x(\beta_*\mu,\mu)$ is the simulated sample mean,
is at least $\ell$; this is also equivalent to say that
\begin{equation}\label{eq:confidence-interval-condition-for-simulated-sample-means}
x\frac{\beta_*+1}{\beta_*-1}-\sqrt{\frac{8\beta_0x}{(\beta_0-1)^3}}\frac{\Phi^{-1}\left(\frac{1+\ell}{2}\right)}{\sqrt{\mu}}
<\overline{A}_x(\beta_*\mu,\mu)<
x\frac{\beta_*+1}{\beta_*-1}+\sqrt{\frac{8\beta_0x}{(\beta_0-1)^3}}\frac{\Phi^{-1}\left(\frac{1+\ell}{2}\right)}{\sqrt{\mu}}.
\end{equation}

Thus, when $\mu$ is large, $\beta$ can be estimated by the following
items.
\begin{itemize}
\item The confidence interval at the level $\ell$, when $x<\overline{A}_x(\beta_*\mu,\mu)-
\sqrt{\frac{8\beta_0x}{(\beta_0-1)^3}}\frac{\Phi^{-1}\left(\frac{1+\ell}{2}\right)}{\sqrt{\mu}}$:
\begin{equation}\label{eq:confidence-interval-for-beta}
\left(\frac{\overline{A}_x(\beta_*\mu,\mu)+\sqrt{\frac{8\beta_0x}{(\beta_0-1)^3}}\frac{\Phi^{-1}\left(\frac{1+\ell}{2}\right)}{\sqrt{\mu}}+x}
{\overline{A}_x(\beta_*\mu,\mu)+\sqrt{\frac{8\beta_0x}{(\beta_0-1)^3}}\frac{\Phi^{-1}\left(\frac{1+\ell}{2}\right)}{\sqrt{\mu}}-x},
\frac{\overline{A}_x(\beta_*\mu,\mu)-\sqrt{\frac{8\beta_0x}{(\beta_0-1)^3}}\frac{\Phi^{-1}\left(\frac{1+\ell}{2}\right)}{\sqrt{\mu}}+x}
{\overline{A}_x(\beta_*\mu,\mu)-\sqrt{\frac{8\beta_0x}{(\beta_0-1)^3}}\frac{\Phi^{-1}\left(\frac{1+\ell}{2}\right)}{\sqrt{\mu}}-x}\right).
\end{equation}
\item The point estimation:
\begin{equation}\label{eq:point-estimation-for-beta}
\frac{\overline{A}_x(\beta_*\mu,\mu)+x}{\overline{A}_x(\beta_*\mu,\mu)-x}.
\end{equation}
\end{itemize}
Moreover, in addition to these estimators, the meaningful information carried by
these simulations concern both $\mu_0$ and $\beta_*$ for which the inequality
\eqref{eq:confidence-interval-condition-for-simulated-sample-means} is satisfied.

Now we are ready to present some numerical values. In all cases we perform
simulations by setting $x=1$ and $\beta_0=1.25$; furthermore, the size of
simulated sample paths is $10^3$ and the confidence level is $\ell=0.95$.

\begin{table}[h!]
	\begin{center}
		\caption{Numerical approximations for the confidence interval for $\beta$ varying $\alpha$}
		\label{tab:table1}
		\begin{tabular}{||c||c|c|c|c|c|c} 
			\hline
			$\alpha$ & $\mu$ &$\beta_*$ & $x\frac{\beta_*+1}{\beta_*-1}$&$\overline{A}_x(\beta_* \mu, \mu)$
			&confidence interval \eqref{eq:confidence-interval-for-beta}&point estimation \eqref{eq:point-estimation-for-beta}\\
			\hline
			0.7 & 1000 & 1.75 & 3.${\overline 6}$ & 5.15575 &(1.349423,1.772864)& 1.481261\\
			\hline
			0.8 & 1000 & 1.75 & 3.${\overline 6}$ & 4.496909 &(1.394876,2.03684)& 1.571934\\
			\hline
			0.9 & 1000 & 1.75 & 3.${\overline 6}$ & 4.03037 &(1.434939,2.367616)& 1.659985\\
			\hline
			0.925 & 1000 & 1.75 & 3.${\overline 6}$ & 3.92666 &(1.444975,2.472007)& 1.683373\\
			\hline
		\end{tabular}
	\end{center}
\end{table}

\begin{table}[h!]
	\begin{center}
		\caption{Numerical approximations for the confidence interval for $\beta$ varying $\mu$}
		\label{tab:table2}
		\begin{tabular}{c||c||c|c|c|c|c} 
			\hline
			$\alpha$ & $\mu$ &$\beta_*$ & $x\frac{\beta_*+1}{\beta_*-1}$&$\overline{A}_x(\beta_* \mu, \mu)$
			&confidence interval \eqref{eq:confidence-interval-for-beta}&point estimation \eqref{eq:point-estimation-for-beta}\\
			\hline
			0.9 & 1000 & 2 & 3 & 3.297049 &(1.517462,3.743193)& 1.870682\\
			\hline
			0.9 & 5000 & 2 & 3 & 3.292031 &(1.66817,2.257219)& 1.872588\\
			\hline
			0.9 & 10000 & 2 & 3 & 3.292868 &(1.717179,2.112946)& 1.87227\\
			\hline
			0.9 & 20000 & 2 & 3 & 3.291491 &(1.756974,2.030459)& 1.872794\\
			\hline
		\end{tabular}
	\end{center}
\end{table}

\begin{table}[h!]
	\begin{center}
		\caption{Numerical approximations for the confidence interval for $\beta$ varying $\beta_*$}
		\label{tab:table3}
		\begin{tabular}{c|c||c||c|c|c|c} 
			\hline
			$\alpha$ & $\mu$ &$\beta_*$ & $x\frac{\beta_*+1}{\beta_*-1}$&$\overline{A}_x(\beta_* \mu, \mu)$
			&confidence interval \eqref{eq:confidence-interval-for-beta}&point estimation \eqref{eq:point-estimation-for-beta}\\
			\hline
			0.8 & 1000 & 1.5 & 5 & 6.128791                      &(1.298652,1.561668)& 1.389955\\
			\hline
			0.8 & 1000 & 2 &3 & 3.678367                         &(1.470994,2.80116)& 1.746724\\
			\hline
			0.8 & 1000 & 2.5 & 2.${\overline 3}$ & 2.860925      &(1.583278,7.827018)& 2.074734\\
			\hline
			0.8 & 1000 & 2.75 & 2.$\overline{142857}$ & 2.626983 &(1.625987,34.89173)& 2.229269\\
			\hline
		\end{tabular}
	\end{center}
\end{table}


We conclude with some comments, and we also refer to the figures presented
below. In Table \ref{tab:table1}, we have increasing values of $\alpha$ and
we find decreasing values of the sample mean $\overline A_x(\beta_* \mu, \mu)$
(as expected) that tend to the asymptotic value $x\frac{\beta_*+1}{\beta_*-1}$.
Furthermore, we find also increasing values of the point estimate
\eqref{eq:point-estimation-for-beta} and wider confidence intervals.
In Table \ref{tab:table2}, for increasing values of $\mu$, we find quite stable
values for $\overline A_x(\beta_* \mu, \mu)$ and the point estimation
\eqref{eq:point-estimation-for-beta}. Moreover we obtain more and more narrow
confidence intervals as $\mu$ increases. In Table \ref{tab:table3}, it appears
evident that, for increasing values of $\beta_*$, the values of the sample mean
$\overline A_x(\beta_* \mu, \mu)$ become more accurate estimations for the
corresponding values of $x\frac{\beta_*+1}{\beta_*-1}$, whereas the right
endpoints of the confidence intervals are less accurate. We also remark that,
in all tables, the estimated values based on the point estimation
\eqref{eq:point-estimation-for-beta} are less than the corresponding set values
of $\beta_*$.

From Tables \ref{tab:table1}-\ref{tab:table3}, and all performed  simulations
results, we can say that the numerical strategy to obtain the above estimates is
reliable for high values of $\alpha$. This is easily understandable because the
above estimates are reliable in a neighborhood of the asymptotic value
$x\frac{\beta_*+1}{\beta_*-1}$ (for high value of $\mu$, i.e. for high rate of
downward steps) or, in some sense equivalently, for high value of absorbing
probability $\alpha$ (compare Figure \ref{fig_mu} and left side of Figure
\ref{fig_alfa_beta}). Furthermore, as far as the value $\mu_0$ is concerned (i.e.
the value such that we can obtain reliable estimates, at the confidence level 0.95,
when $\mu>\mu_0$), we can take $\mu_0=1000$. Furthermore, the results in Table
\ref{tab:table2} show that the approximation of the confidence interval improves
as $\mu$ increases. Finally we also stress that all these numerical values provide
indications on the true value of $\beta$ under the scaling 2 for finite values of
$\mu$ (instead of asymptotic results as $\mu\to\infty$).

We conclude with some brief comments on Figures \ref{fig_mu}-\ref{fig_alfa_beta}.
They show that sample paths of the process for different choices of values for
parameters $\mu, \, \alpha, \beta$. In Figure \ref{fig_mu} it is possible to observe
how the paths change for different values of $\mu$. In Figure \ref{fig_alfa_beta}
we consider different values of $\alpha$ and $\beta$; in particular, we set
$\mu=10$ because the effect of different values of $\alpha$ and $\beta$ on the
sample paths appears more evident.

\begin{figure}
\centering
\includegraphics[scale=0.6]{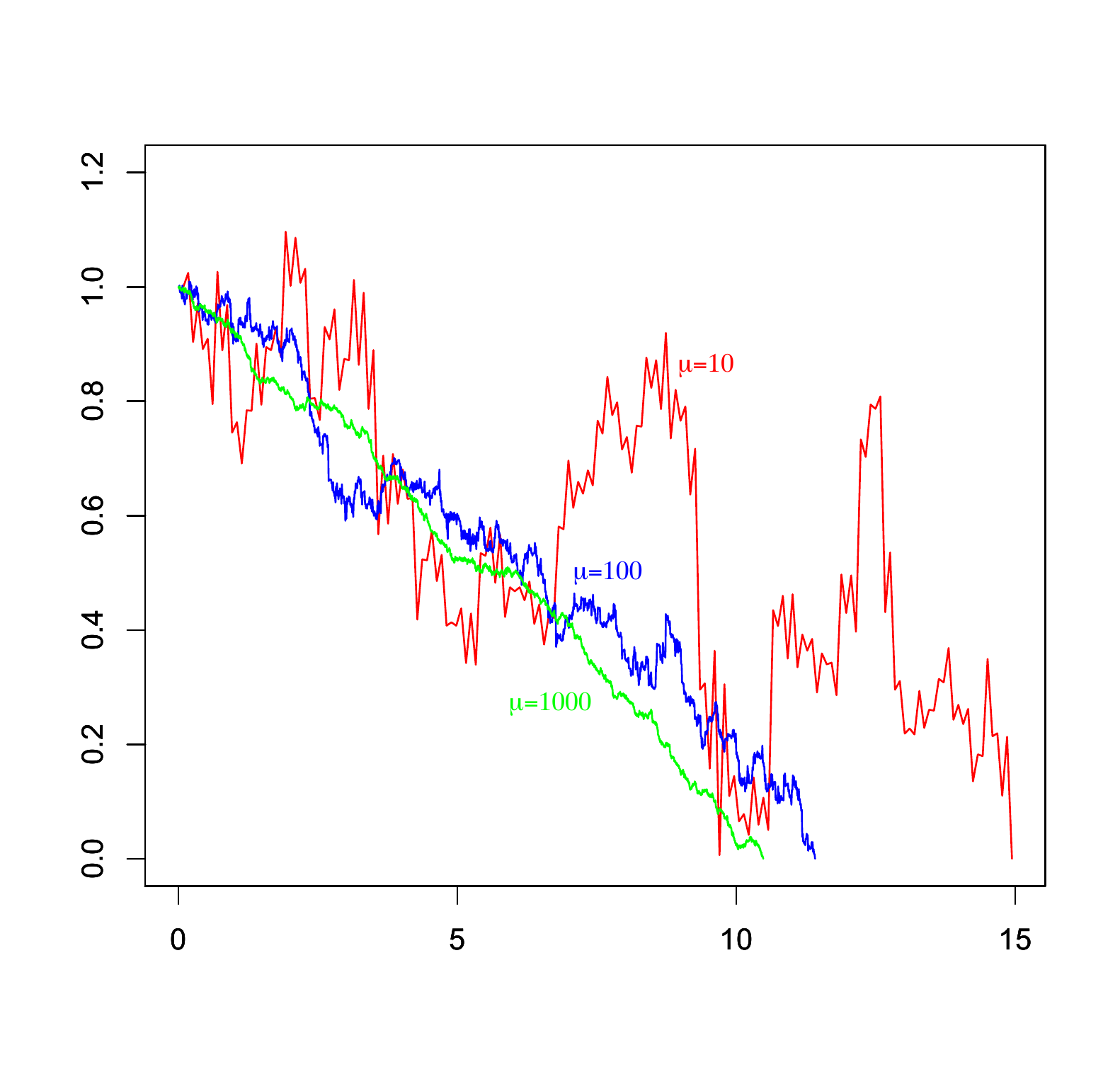}
\vspace*{-1cm}
\caption{Sample paths for different values of $\mu$, for $\alpha=0.9$ and $\beta=1.25$.}
\label{fig_mu}
\end{figure}

\begin{figure}
\centering
\includegraphics[scale=0.42]{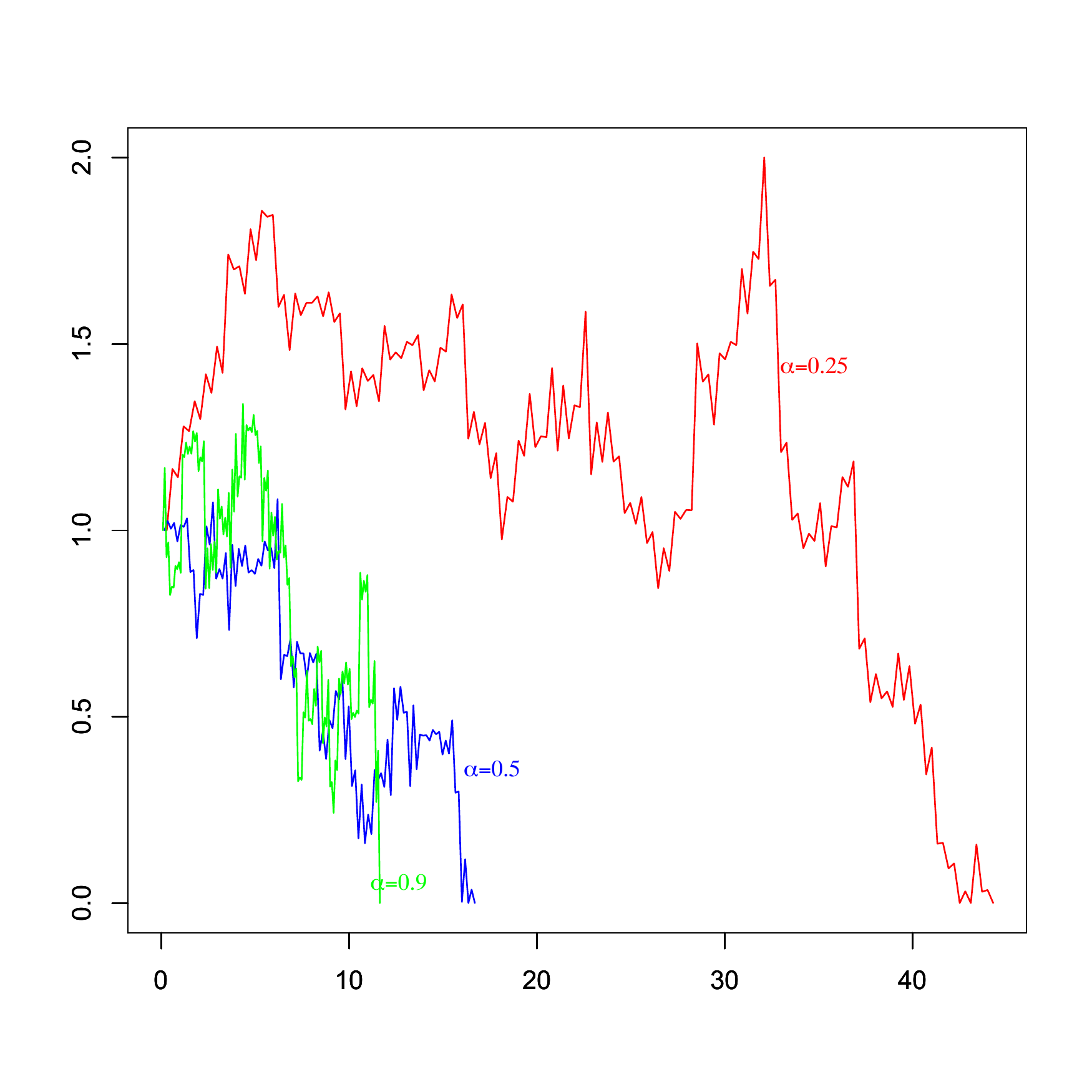}
\includegraphics[scale=0.42]{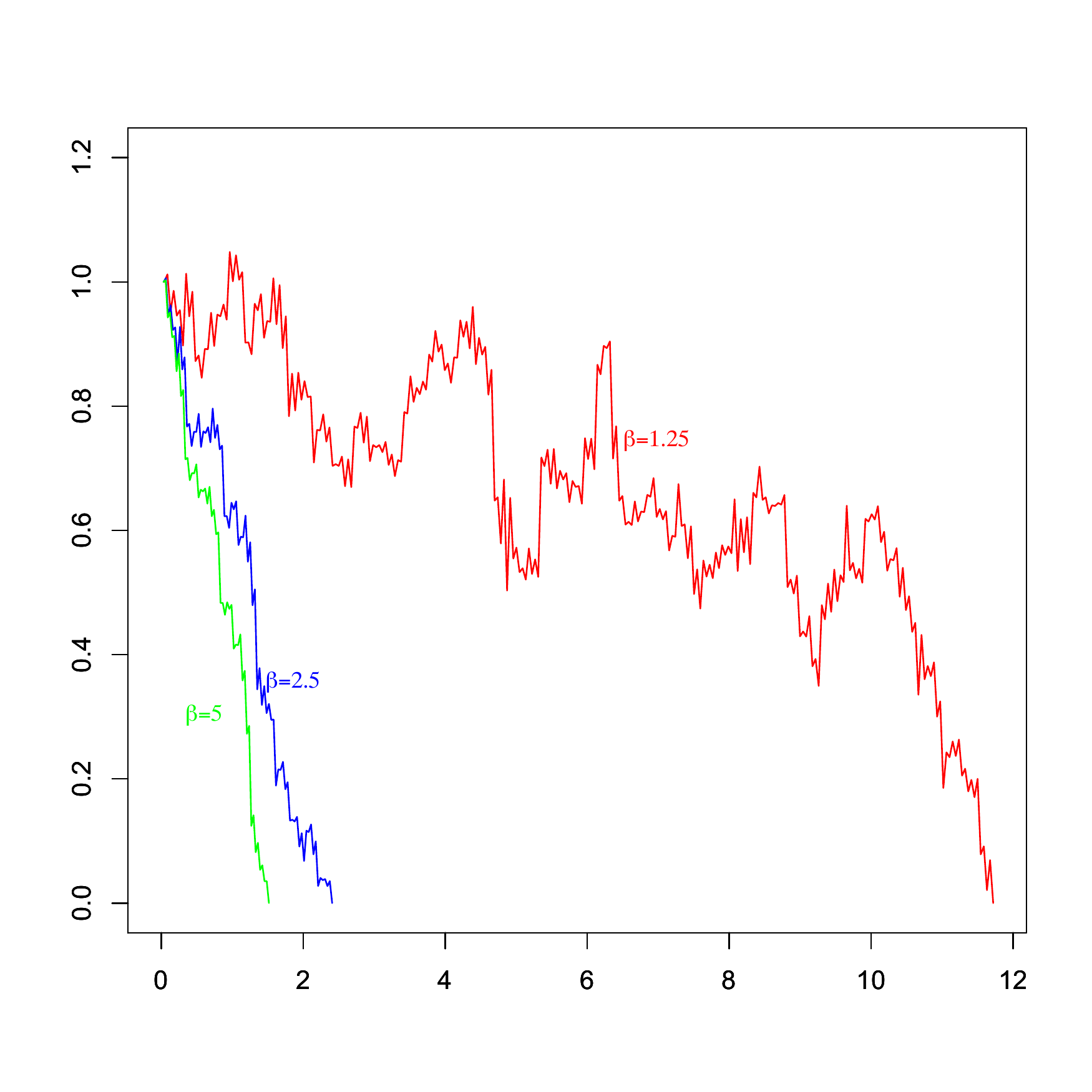}
\caption{Left: sample paths for different values of $\alpha$, for $\mu=10$ and $\beta=1.25$.
Right: the same varying $\beta$, for $\mu=20$ and $\alpha=0.9$.}\label{fig_alfa_beta}
\end{figure}

\end{document}